\documentclass[10pt]{extarticle}

\title{} \author{} \date{}

\usepackage[margin=3cm]{geometry}
\usepackage{mathtools}
\usepackage{amssymb}
\usepackage{amsthm}
\usepackage{microtype} 

\usepackage{xr}
\usepackage{hyperref} 
\usepackage{cleveref}
\hypersetup{
    colorlinks,
    linkcolor = blue,
    citecolor = blue
}

\usepackage[
   backend=biber,
   style=alphabetic,
   sorting=ynt
]{biblatex}
\usepackage[none]{hyphenat}
\usepackage{enumitem}
\setlist[itemize]{topsep=0pt}
\setlist[enumerate]{nosep, topsep=0pt}

\makeatletter
\newtheoremstyle{indented}
  {\topsep}
  {0pt}
  {\addtolength{\@totalleftmargin}{2em}
   \addtolength{\linewidth}{-2em}
   \parshape 1 2em \linewidth}
  {}
  {\bfseries}
  {.}
  {0.5em}
  {}
\makeatother

\theoremstyle{indented}
\newtheorem{theorem}{Theorem}[section]

\theoremstyle{indented}
\newtheorem{definition}[theorem]{Definition}

\theoremstyle{indented}
\newtheorem{lemma}[theorem]{Lemma}

\theoremstyle{indented}
\newtheorem{corollary}[theorem]{Corollary}

\theoremstyle{indented}
\newtheorem{example}[theorem]{Example}

\theoremstyle{indented}
\newtheorem{remark}[theorem]{Remark}

\newcommand{\mbb}[1]{\mathbb{#1}}

\newcommand{\mc}[1]{\mathcal{#1}}

\renewcommand{\L}{\mbb{L}}
\renewcommand{\P}{\mbb{P}}
\newcommand{\C}{\mbb{C}}
\newcommand{\R}{\mbb{R}}
\newcommand{\Q}{\mbb{Q}}
\newcommand{\N}{\mbb{N}}

\newcommand{\wed}{\wedge}

\newcommand{\ua}{\ \mathord{\uparrow} \ }
\newcommand{\da}{\ \mathord{\downarrow} \ }

\newcommand{\alp}{\alpha}

\newcommand{\gam}{\gamma}
\newcommand{\lam}{\lambda}

\renewcommand{\to}{\longrightarrow}

\newcommand{\lrsbig}[1]{\big[ #1 \big]}

\DeclarePairedDelimiter{\abs}{\lvert}{\rvert}

\renewcommand{\int}{\text{o}}

\newcommand{\comment}[1]{}

\usepackage{xcolor}
\colorlet{Sepia}{red!50!black}

\begin{document}

    \title{The Riesz-Kantorovich formulas for $\L$-vector lattices}

    \author{Tomas Chamberlain and Marten Wortel}

    \maketitle

    \abstract{Let $\L$ be a Dedekind complete unital $f$-algebra. We prove the Riesz-Kantorovich formulas for order bounded $\L$-module homomorphisms from a directed partially ordered $\L$-module with the Riesz Decomposition Property into a Dedekind complete $\L$-vector lattice satisfying an additional mild condition.}

    \setcounter{section}{0}

    \section{Introduction}

    Let $\L$ be a real or complex Dedekind complete unital $f$-algebra. In \cite{LF}, a general theory of $\L$-functional analysis was developed, investigating $\L$-normed spaces, i.e., $\L$-modules equipped with an $\L$-valued norm. This theory was mainly inspired by two recent developments. Firstly, the theory of Kaplansky-Hilbert modules (developed by Kaplansky in 1953) was successfully applied in \cite{EHK24} to obtain a functional-analytic proof of the famous Furstenberg-Zimmer structure theorem of ergodic theory. Kaplansky-Hilbert modules are $\L$-Hilbert spaces where $\L$ is an abelian AW*-algebra, which is a complex Dedekind complete unital vector lattice. Secondly, a comprehensive theory of probability theory in vector lattices was recently developed (see for example \cite{LW10, KKW24} and references therein) based on a conditional expectation operator $T$ on a vector lattice $E$. In this theory certain spaces $L^p(T)$ are of crucial importance, and these are $\L$-normed spaces with $\L = R(T)$, a real universally complete (and hence Dedekind complete) vector lattice.

    In both of these motivating examples, the corresponding $\L$-normed spaces have additional structure. Indeed, in \cite{EHK24} the Kaplansky-Hilbert module investigated is $\mathrm{L}^2( \mathrm{X} | \mathrm{Y})$ as a complex $\mathrm{L}^\infty(\mathrm{Y})$-module (here $\mathrm{X}$ and $\mathrm{Y}$ are certain probability spaces), and $\mathrm{L}^2( \mathrm{X} | \mathrm{Y})$ has a natural modulus turning it into a complex vector lattice. The $R(T)$-module $L^p(T)$ is also naturally a real vector lattice.

    It therefore seems desirable to also develop a theory of $\L$-vector lattices, i.e., $\L$-modules that are also lattices compatible with the module structure. This paper is a first step towards developing such a theory. In the classical case (i.e., if $\L = \R$ or $\L = \C$) the theory of complex vector lattices is developed based on real vector lattice theory, so in this paper we only consider the case where $\L$ is a real Dedekind complete unital vector lattice.

    One of the most important results in classical vector lattice theory is the Riesz-Kantorovich formulas, and the main goal of this paper is to develop the $\L$-valued version in full generality. Note that in \cite[Lemma~3.2]{KKW24} these formulas were proven in the very specific case of $L^2(T)^*$ (the order bounded $R(T)$-module homomorphisms from $L^2(T)$ to $R(T)$) using techniques tailored towards that setting.  \\

    The structure of this paper is as follows. In \Cref{s:preliminaries} we explain the basics of $\L$ as well as partially ordered $\L$-modules and $\L$-vector lattices. \Cref{s:archimedean} contains a thorough investigation of Archimedean properties, which turn out to be more involved than in the classical case. The classical Riesz-Kantorovich formulas rely on an extension lemma for additive maps between cones, and in \Cref{s:extension} we develop its $\L$-valued analogue. Even in the classical case, our version of this lemma is a slight improvement of the classical result, since the codomain is only required to be almost Archimedean instead of Archimedean. Finally, in \Cref{s:RK} we prove the Riesz-Kantorovich formulas (\Cref{thm-Riesz-Kantorovich}).

    \section{Preliminaries}\label{s:preliminaries}

    In this paper, $\L$ is an arbitrary Dedekind-complete unital $f$-algebra. We use the phrase ``in the classical case'' to refer to the particular case where $\L=\R$. Since $\L$ is a unital algebra over $\R$, we identify each real number $r$ with its image $r \cdot 1$ in $\L$, and hence consider $\R$ as a sub-lattice and sub-algebra of $\L$. In particular, we denote the multiplicative unit of $\L$ by $1$. The symbol $\P$ denotes the set $\{\pi\in\L : \pi^2 = \pi\}$ of idempotents in $\L$. Under the order inherited from $\L$, $\P$ is a sublattice of $\L$ and a complete Boolean algebra with least element $0$, greatest element $1$, meet operation $(\pi,\rho) \mapsto \pi\rho$, and complementation operation $\pi \mapsto \pi^c := 1 - \pi$. $\P$ is isomorphic to the Boolean algebra of bands of $\L$; for every band of $\L$, the corresponding element of $\P$ is the image of $1$ under the band projection, and for every idempotent $\pi$, the corresponding band is $\pi\L$.

    As explained in \cite[\S 2.1]{LF}, there exists a unique Stonean space $K$ such that $\L$ is (isomorphic to) an order-dense vector lattice ideal and sub-algebra of $C_\infty(K)$, and furthermore $C(K) \subseteq \L$. Using this representation, the elements of $\P$ are the $\{0,1\}$-valued functions in $C(K)$, and they are precisely the indicator functions of the clopen subsets of $K$. The real numbers, as elements of $\L$, correspond to the constant functions in $C(K)$. Note that $\L = \R$ if and only if $\P$ is the two-element Boolean algebra. We use the notation $\L_1$ to denote the vector lattice ideal of $\L$ generated by $1$. Under the representation of $\L$ using $C_\infty(K)$, $\L_1$ corresponds to $C(K)$.

    \begin{definition}
        Let $X$ be an $\L$-module and $x\in X$. The \textbf{support} of $x$ is denoted by $\pi_x$ and is defined as $\pi_x := \inf \{\pi \in \P : \pi x = x\}$.
    \end{definition}

    Note that $\pi_x\in\P$. For $\lam \in \L$ with $\L$ considered as a module over itself, $\pi_\lam$ is the indicator function of the closure of the set $\{k\in K: \lam(k) \neq 0\}$. Furthermore, $\pi_\lam \lam = \lam$.

    The representation of $\L$ explained above makes it clear that, for arbitrary $\lam$ and $\mu$ in $\L$, the following are equivalent: $\lam\mu = 0$; $\abs{\lam} \wed \abs{\mu} = 0$; $\pi_\lam \wed \pi_\mu = 0$; and $\pi_\lam \pi_\mu = 0$.

    An element $x$ of an $\L$-module is said to be \textbf{support-attaining} if $\pi_x x = x$, and an $\L$-module itself is said to be support-attaining if each of its elements is support-attaining. The following example shows that not every $\L$-module is support-attaining.\footnote{It can be shown that an $\L$-module is support-attaining if and only if every cyclic submodule is projective, if and only if the module is non-singular.}

    \begin{example}\label{ex-not-SA}
        Let $\L = \ell^\infty$, $X = \ell^\infty/c_0$, and $x\in X$. For each $n\in\N$, let $\pi_n$ be the indicator function of the set $\N\setminus\{1,\ldots,n\}$. Then $\pi_n x = x$ for all $n\in\N$, but $\pi_n \da 0$, so $\pi_x = 0$.
    \end{example}

    \begin{definition}
        Let $X$ be an $\L$-module. We say that $X$ is a \textbf{partially ordered $\L$-module} (or \textbf{PO $\L$-module}) if $X$ has a partial order such that, for all $x,y,z \in X$ and all $\lam\in\L$, \begin{enumerate}
            \item $x \leq y$ implies $x+z \leq y+z$, and
            \item $0 \leq \lam$ and $x\leq y$ imply $\lam x \leq \lam y$.
        \end{enumerate}
        The \textbf{positive cone} of $X$ is the set $X^+ := \{ x\in X: 0 \leq x \}$. If $X$ is a PO $\L$-module, then we say that $X$ is \textbf{directed} if every pair of elements of $X$ has an upper bound. We say that $X$ is an \textbf{$\L$-vector lattice} if the partial order on $X$ is a lattice order.
    \end{definition}

    As in the classical case, the partial orders on a given $\L$-module which make it into a PO $\L$-module are in bijection with the cones of $X$, i.e., subsets of $X$ which are closed under addition and positive scaling and contain no non-trivial submodules.

    A PO $\L$-module $X$ is directed if and only if the submodule generated by $X^+$ is equal to $X$. We will refer to directed PO $\L$-modules as \textbf{directed $\L$-modules}.

    Let $X$ and $Y$ be PO $\L$-modules. An $\L$-module homomorphism from $X$ to $Y$ is called an \textbf{operator} from $X$ to $Y$, and the set of operators is denoted by $\mc{L}(X,Y)$. An operator is said to be \textbf{positive} if it maps $X^+$ to $Y^+$. We make $\mc{L}(X,Y)$ a PO $\L$-module by letting the positive cone be the set of positive operators. An operator is said to be \textbf{regular} if it is the difference of two positive operators, and the set of regular operators is denoted $\mc{L}_\text{r}(X,Y)$. An operator $T$ is said to be \textbf{order-bounded} if, for every order interval $[x,y]\subseteq X$, $T([x,y])$ is contained in an order interval in $Y$. The set of order-bounded operators is denoted $\mc{L}_\text{b}(X,Y)$. As in the classical case, every regular operator is order-bounded. The expression $X^\sim$ denotes the set of order-bounded operators from $X$ to $\L$.

    For $x,y$ in an $\L$-vector lattice, let $\pi_{x<y}$ denote the support of $(y-x)^+$. In particular, for $\lam,\mu\in\L$, $\pi_{\lam<\mu}$ is the indicator function of the closure of the set $\{k\in K: \lam(k) < \mu(k)\}$. It follows that $\pi_{\lam<\mu} \lam \leq \pi_{\lam<\mu}\mu$, and this gives the following lemma.

\begin{lemma}\label{lem-bounded-seq-converging-to-lambda}
    Let $\lam \in \L$. Then there exists a sequence $\pi_n \ua 1$ in $\P$ such that $-n \leq \pi_n \lam \leq n$ for all $n\in\N$. Furthermore, if $\lam\in\L^+$, then $\pi_n \lam \ua \lam$.
\end{lemma} \begin{proof}
    For each $n\in\N$, let $\pi_n = \pi_{\abs{\lam}<n}$. It is clear from the representation of $\L$ that $\pi_n \ua 1$, and $-n \leq \pi_n \lam \leq n$ for all $n\in\N$.
\end{proof}

\begin{lemma}\label{lem-P-preserves-intervals}
    For a PO $\L$-module $X$ with $x\in X^+$ and $\pi\in\P$, we have $[0,\pi x] = \pi [0,x]$.
\end{lemma} \begin{proof}
    For any $y \in [0,x]$, $\pi y \in [0,\pi x]$, so $\pi [0,x] \subseteq [0,\pi x]$. Conversely, suppose $z \in [0,\pi x]$. Then $z \in [0,x]$, and $\pi^c z \in [0, \pi^c \pi x] = \{0\}$, so $(1-\pi) z = \pi^c z = 0$. Thus $z = \pi z \in \pi [0,x]$, which shows that $[0,\pi x] \subseteq \pi [0,x]$.
\end{proof}

If $X$ is a PO $\L$-module, $x \in X$, and $D \subseteq X$, then the expression $x \leq D$ means that $x$ is a lower bound of $D$.

\begin{lemma}\label{lem-scalar-mult-continuity-right}
    Let $X$ be a PO $\L$-module. If $\lam\in\L^+$ and $D \da x$ in $X$, then $\lam D \da \lam x$.
\end{lemma} \begin{proof}
    First suppose $x = 0$. Clearly $0 \leq \lam D$. Let $y \leq \lam D$. Then, for all $d\in D$, $y \leq y + d \leq \lam d + d = (\lam + 1)d$. Since $\lam+1 \geq 1$, it is clear from the representation of $\L$ that $\lam+1$ is invertible, and its inverse is positive. This gives $(\lam+1)^{-1}y \leq d$ for all $d\in D$. Hence $(\lam+1)^{-1}y \leq 0$, and multiplying by $\lam+1$ gives $y \leq 0$.

    For arbitrary $x\in X$, observe that $D \da x$ if and only if $D-x \da 0$. Applying the above result gives $\lam D - \lam x \da 0$, and hence $\lam D \da \lam x$.
\end{proof}
    
    \section{Archimedean Properties}\label{s:archimedean}

We begin with the main definitions.

\begin{definition}\label{def-archimedean}
    Let $X$ be a PO $\L$-module and $S \in \{\L,\R,\P\}$. Consider the following conditions. \begin{itemize}
        \item[(i)] For all $D \subseteq S$ and all $x\in X^+$, if $D \da 0$ in $\L$, then $Dx \da 0$ in $X$.
        \item[(ii)] For all $D \subseteq S$ and all $x\in X^+$, if $D \da 0$ in $\L$, $y \in X$, and $-Dx \leq y \leq Dx$, then $y = 0$.
    \end{itemize} We say that $X$ is \textbf{$S$-Archimedean} if it satisfies condition $(i)$. We say that $X$ is \textbf{almost $S$-Archimedean} if it satisfies condition $(ii)$.
    
    We say that $X$ is \textbf{Archimedean} if it is $\L$-Archimedean, and we say that $X$ is \textbf{sequentially} \textbf{$S$-Archimedean} (resp. \textbf{sequentially almost $S$-Archimedean}) if it satisfies condition $(i)$ (resp. condition $(ii)$) with $D$ replaced by an arbitrary sequence in $S$ which decreases to zero.
\end{definition}

In the classical case, the only idempotents are $0$ and $1$, and so the $\P$-Archimedean condition is trivially satisfied.

Note that the ``$\R$-Archimedean'' and ``almost $\R$-Archimedean'' properties are equivalent to their sequential versions. As in the classical case, each Archimedean property implies the corresponding ``almost'' version, and for an $\L$-vector lattice, each Archimedean property is equivalent to the corresponding ``almost'' version.

In \cite[\S 2]{ideal-theory-in-f-algebras}, it is shown that multiplication by a positive element in an $f$-algebra is an orthomorphism, and thus is order-continuous. Hence, $\L$, considered as a PO module over itself, is Archimedean.

A directed $\L$-module $X$ is $\R$-Archimedean if and only if, for all $x,y\in X$, if $x \leq \frac{1}{n}y$ for all $n\in\N$, then $x \leq 0$. This is the classical definition of an Archimedean space. In the classical case, a directed $\L$-module is Archimedean if and only if it is $\R$-Archimedean. In the general case, an Archimedean PO $\L$-module is always $\R$-Archimedean, but the two properties are not equivalent. Indeed, Example \ref{ex-not-SA} gives an $\R$-Archimedean $\L$-module that is not Archimedean. But notice that, in that example, the failure of $X$ to be Archimedean is exhibited by a net of idempotents. This turns out to be the only source of counterexamples, as we will now show.

\begin{theorem}\label{thm-P-R-arch-implies-arch}
    Let $X$ be a PO $\L$-module. Then $X$ is Archimedean if and only if it is both $\R$-Archimedean and $\P$-Archimedean.
\end{theorem}
\begin{proof}
    Clearly, if $X$ is Archimedean, then it is $\R$-Archimedean and $\P$-Archimedean. Now assume $X$ is $\R$-Archimedean and $\P$-Archimedean. Let $x \in X^+$ and let $D\da 0$ in $\L$. We wish to show that $\inf (Dx) = 0$, so let $y \leq Dx$.

    Fix $r \in \R$ with $0 < r$. For each $\lam\in D$, we have $\lam + r = \lam\wed r + \lam\vee r$, giving $\lam = \lam \wed r + (\lam \vee r - r) = \lam \wed r + (\lam-r)^+$. Note also that the function $\lam \mapsto \pi_{\lam>r}$ (with domain $D$) is order-preserving.
    
    We claim that $\inf_{\lam\in D} \pi_{\lam>r} = 0$. To prove this, let $\sigma := \inf_{\lam\in D} \pi_{\lam>r}$ and let $\mu\in D$. Then $0\leq \sigma \leq \pi_{\mu>r} = \pi_{(\mu-r)^+}$, and since $(\mu-r)^-$ and $\pi_{(\mu-r)^+}$ are disjoint, $\sigma$ is disjoint from $(\mu-r)^-$. Hence $(\mu-r)\sigma = (\mu-r)^+\sigma \geq 0$, and so $r\sigma \leq \mu\sigma$. Taking the infimum over all $\mu\in D$ gives $r\sigma\leq \inf (D \sigma) = 0$, and dividing by $r$ gives $\sigma \leq 0$. Therefore $\sigma = 0$, as claimed.

    We now show that $y \leq rx$. For each $\lam,\mu \in D$, let $\nu \in D$ be a lower bound of $\{\lam,\mu\}$. (Such a $\nu$ exists because $D$ is downward-directed.) Then $y \leq (\nu\wed r)x + (\nu-r)^+ x \leq rx + (\nu-r)^+x$, so \begin{align*}
        y-rx &\leq (\nu-r)^+x \\
        &= \pi_{\nu>r}(\nu-r)^+ x \\
        &\leq \pi_{\lam>r}(\mu-r)^+ x.
    \end{align*} Since $X$ is $\P$-Archimedean and $\inf_{\lam\in D} \pi_{\lam>r} = 0$, we have $\inf_{\lam\in D} (\pi_{\lam>r} (\mu-r)^+x) = 0$, which implies that $y -rx \leq 0$.

    Hence, $y \leq rx$ for all $r\in\R_{>0}$. Since $X$ is $\R$-Archimedean, $y \leq 0$.
\end{proof}

The proof of the next theorem is a simple modification of the proof of Theorem \ref{thm-P-R-arch-implies-arch}.

\begin{theorem}\label{thm-P-R-arch-implies-arch-seq}
    Let $X$ be a PO $\L$-module. Then $X$ is sequentially Archimedean if and only if it is sequentially $\R$-Archimedean and sequentially $\P$-Archimedean. \qed
\end{theorem}

Recall that a poset is said to be \textbf{Dedekind-complete} if every non-empty subset with an upper bound has a supremum and every non-empty subset with a lower bound has an infimum. (Equivalently, every non-empty subset with an upper bound has a supremum.) The following theorem can be proven exactly as in the classical case.
\begin{theorem}
    Every Dedekind-complete $\L$-vector lattice is $\R$-Archimedean. \qed
\end{theorem}
However, a Dedekind-complete $\L$-vector lattice need not be sequentially almost $\P$-Archimedean.

\begin{example}
    Suppose $\L = \ell^\infty$. Then $K = \beta\N$, the Stone-\v{C}ech compactification of the natural numbers. Let $k \in \beta\N\setminus\N$, and let $I \subseteq \L$ be the (ring- and lattice-) ideal of elements which vanish at $k$. It can be shown that the quotient ring $\L/I$ is ring-isomorphic to $\R$ and that the quotient map is order-preserving, so $\L/I$ is in fact a Dedekind-complete totally-ordered $\L$-module. If we define the sequence $\pi_n$ as in Example \ref{ex-not-SA}, and if $[1]$ denotes the equivalence class in $\L/I$ of the constant $1$ sequence, then we see that $\pi_n[1] = [1]$ for all $n\in\N$. Hence $-\pi_n[1] \leq [1] \leq \pi_n [1]$ for all $n\in\N$ and $\pi_n \da 0$, but of course $[1]\neq [0]$, showing that $X$ is not sequentially almost $\P$-Archimedean.
\end{example}

The $\P$-Archimedean property is closely connected to support-attainment. In fact, if $X$ is a directed $\L$-module, then it can be shown that $X$ is support-attaining if and only if it is almost $\P$-Archimedean.

\begin{theorem}
    Let $X$ be a directed $\L$-module. Then $X$ is support-attaining if and only if it is almost $\P$-Archimedean.
\end{theorem} \begin{proof}
    Suppose $X$ is almost $\P$-Archimedean. Let $x \in X$ and, since $X$ is directed, let $y \in X$ be an upper bound of $\{x,-x\}$. For each $\pi\in\P$ satisfying $\pi x = x$, we have $\pi-\pi_x \geq 0$ and thus $-(\pi-\pi_x)y \leq (\pi-\pi_x)x = x - \pi_x x \leq (\pi-\pi_x)y$. Since the set $\{\pi-\pi_x: \pi \in \P, \pi x = x\} \subseteq \P$ decreases to zero in $\L$ and $X$ is almost $\P$-Archimedean, we have $x-\pi_x x = 0$, and thus $x = \pi_x x$.

    Now suppose $X$ is support-attaining. Let $x \in X^+$ and $y\in X$ with $D \da 0$ in $\P$ and $-\pi x \leq y \leq \pi x$ for all $\pi\in D$. It follows that $\pi_y \leq \pi$ for all $\pi\in D$. Indeed, since $-\pi x \leq y \leq \pi x$, multiplying by $\pi^c$ gives $0 = -\pi^c \pi x \leq \pi^c y \leq \pi^c \pi x = 0$, which implies $\pi^c y = 0$ and thus $\pi y = y$. Since $D \da 0$, it follows that $\pi_y = 0$, and hence, by support-attainment, $y = \pi_y y = 0$, showing that $X$ is almost $\P$-Archimedean.
\end{proof}

    \section{Extension Lemma}\label{s:extension}

To prove the Riesz Kantorovich Theorem, we need an extension lemma. Classically, we have the following (given, for instance, as Lemma 1.26 in \cite{cones_and_duality}).
\begin{lemma}
    Let $X$ and $Y$ be directed $\R$-vector spaces, with $Y$ Archimedean. Let $T : X^+ \to Y^+$ be an additive function. Then $T$ has a unique extension to a positive operator from $X$ to $Y$.
\end{lemma}

Our version of the extension lemma for directed $\L$-modules is similar to the classical version, but there are some differences. Firstly, we note that the proof of Lemma 1.26 in \cite{cones_and_duality} does not actually require $Y$ to be Archimedean; it goes through if $Y$ is merely almost Archimedean. In our case, we assume that $Y$ is sequentially almost Archimedean (that is, sequentially almost $\L$-Archimedean). Secondly, we require the map $T$ to be $\P$-homogeneous. (This condition is trivially satisfied in the classical case.) If $T$ is not $\P$-homogeneous, then it need not admit a linear extension. For example, if $\L = \R^2$ and $X = Y = \L$, then the function $T : \L^+ \to \L^+$ given by $T(x,y) = (y,x)$ is additive, and it admits a unique additive extension $\hat{T} :\L \to \L$ given by the same expression. But $\hat{T}$ is not an operator because it is not $\P$-homogeneous.

Recall that $\L_1$ is the vector lattice ideal of $\L$ generated by $1$. Under the representation of $\L$ using $C_\infty(K)$, $\L_1$ corresponds to the set of elements of $\L$ that are real-valued everywhere.

\begin{theorem}\label{thm-pos-mapping-extension}
    Let $X$ and $Y$ be directed $\L$-modules with $Y$ sequentially almost Archimedean, and let $T: X^+ \to Y^+$ be an additive $\P$-homogeneous function. Then $T$ extends uniquely to a positive operator $\widehat{T}:X\to Y$.
\end{theorem}

\begin{proof}
    We first note that since $Y$ is sequentially almost Archimedean, it is almost $\R$-Archimedean and sequentially almost $\P$-Archimedean.

    As in the classical case, $T$ is order-preserving and $\Q^+$-homogeneous. In particular, by additivity, $\Q^+$-homogeneity, and $\P$-homogeneity, $T$ is homogeneous with respect to any $\Q^+$-linear combination of idempotents. We refer to a $\Q^+$-linear combination of idempotents as a ``$\Q^+$-step function''.

    Suppose we have shown that $T$ is $\L^+$-homogeneous. Then we extend $T$ to $X$ as follows. For each $x\in X$, by the directedness of $X$ there exist $y,z\in X^+$ such that $x = y - z$. Set $\widehat{T}(x) := T(y)-T(z)$. This gives the desired extension. The proof of the fact that $\widehat{T}$ is well defined and unique is the same as in the classical case. Thus, it remains to show that $T$ is $\L^+$-homogeneous.
    
    To show that $T$ is $\L_1^+$-homogeneous, let $\beta\in\L_1^+$ and $x\in X^+$. By the Freudenthal Spectral Theorem \cite[Theorem 2.8]{positive_operators}, there exists a sequence $(\alpha_n)$ of $\Q^+$-step functions such that $\alpha_n \ua \beta$ uniformly relative to $1$ in $\L$. Since $\beta\in\L_1$, there also exists a sequence $(\gamma_n)$ of $\Q^+$-step functions such that $\gamma_n \da \beta$ uniformly relative to $1$.

    Let $0<r\in\R$, and let $n\in\N$ such that $\beta-r < \alp_n$ and $\gam_n < \beta+r$. Then, \[\begin{array}{rcccccccl}
        (\beta-r) T(x) &\leq& \alpha_n T(x) &=& T(\alpha_n x) &&&&\\
        && &\leq& T(\beta x) && \\
        && &\leq& T(\gamma_n x) &=& \gamma_n T(x) &\leq& (\beta+r) T(x).
    \end{array}\] Therefore, $T(\beta x) - \beta T(x) \in \lrsbig{\!-\!r T(x), r T(x)}$. Since $Y$ is almost $\R$-Archimedean, we obtain $T(\beta x) = \beta T(x)$, concluding the proof that $T$ is $\L_1^+$-homogeneous.

    We now show that $T$ is $\L^+$-homegeneous. Let $\lam\in\L^+$. By Lemma \ref{lem-bounded-seq-converging-to-lambda}, there exists a sequence $\pi_n \ua 1$ in $\P$ such that $-n\leq\pi_n\lam\leq n$ for all $n\in\N$. Using the $\P$-homogeneity of $T$ and then the fact that $\pi_n \lam\in\L_1^+$, we have \[
        \pi_n T(\lam x) \ = \ T(\pi_n \lam x) \ = \ \pi_n \lam T(x)
    \] for all $n$. Therefore, we have the following two inequalities: \[\begin{array}{ccccc}
        \pi_n \lam T(x) &=& \pi_n T(\lam x) &\leq& T(\lam x) \\
        -\lam T(x) &\leq& -\pi_n\lam T(x) &=& -\pi_n T(\lam x).
        \end{array}\] Adding the two inequalities gives \[
        (\pi_n-1) \lam T(x) \ \leq \ T(\lam x) - \lam T(x) \ \leq \ (1-\pi_n)T(\lam x).
    \] Since $Y$ is directed, there exists a $y\in Y$ such that $\{\lam T(x), T(\lam x)\} \leq y$, giving \[
        -(1-\pi_n)y \ \leq\ T(\lam x) - \lam T(x) \ \leq\ (1-\pi_n)y
    \] for all $n\in\N$. Since $Y$ is sequentially almost $\P$-Archimedean and $(1-\pi_n) \da 0$, we have $T(\lam x) = \lam T(x)$.
\end{proof}

\begin{remark}\label{rem-pos-extension-L_1}
    In the proof of Theorem \ref{thm-pos-mapping-extension}, to show that $T$ is $\L_1^+$-homogeneous, only the almost $\R$-Archimedean property of $Y$ was used. The sequential almost $\P$-Archimedean property was only used to show that $T$ is homogeneous with respect to elements of $\L^+$ not in $\L_1^+$. Hence, if $\L=\L_1$, then the conclusion of Theorem \ref{thm-pos-mapping-extension} still holds if we weaken the assumption that $Y$ is sequentially almost Archimedean to merely the assumption that $Y$ is almost $\R$-Archimedean.
\end{remark}

    \section{The Riesz-Kantorovich Theorem}\label{s:RK}

Recall that a PO real vector space $X$ is said to have the \textbf{Riesz Decomposition Property (RDP)} if, for all $x$ and $y$ in $X^+$, $[0,x+y] = [0,x]+[0,y]$. It is well-known (for instance, it is shown in \cite[\S 1.8]{cones_and_duality}) that this is equivalent to the condition that, for all $x,y,p,q\in X$, if $\{x,y\}\leq\{p,q\}$, then there exists a $z\in X$ such that $\{x,y\}\leq z \leq \{p,q\}$. This is clearly equivalent to the statement that for all $x,y\in X$, the (possibly empty) set of upper bounds of $\{x,y\}$ in $X$ is downward-directed. The definition of the RDP extends to PO $\L$-modules in the obvious way, and the same equivalences hold. An $\L$-vector lattice is a PO $\L$-module $X$ such that, for all $x,y\in X$, the set of upper bounds of $\{x,y\}$ in $X$ has a least element, so it is clear that every $\L$-vector lattice has the RDP (and is directed). Hence, in particular, the following theorem applies when $X$ is an $\L$-vector lattice.

For an $\L$-vector lattice $Y$, as explained in \cite[\S 3.3]{LF}, if $Y$ admits an $\L$-valued norm (which is also defined there), then it is support-attaining and thus $\P$-Archimedean.

\begin{theorem}\label{thm-Riesz-Kantorovich}[The Riesz-Kantorovich Theorem]
    Let $X$ be a directed $\L$-module with the RDP, and let $Y$ be a Dedekind-complete $\L$-vector lattice. If $Y$ is sequentially $\P$-Archimedean or $\L=\L_1$, then the following hold. \begin{enumerate}
        \item $\mc{L}_\text{b}(X,Y)$ is an $\L$-vector lattice;
        \item $\mc{L}_\text{b}(X,Y) = \mc{L}_\text{r}(X,Y)$;
        \item $\mc{L}_\text{b}(X,Y)$ is Dedekind-complete;
        \item the lattice operations in $\mc{L}_\text{b}(X,Y)$ are given by the following formulas: for $S$ and $T$ in $\mc{L}_\text{b}(X,Y)$, and for all $x\in X^+$, we have \begin{align*}
            (S \vee T)(x) &= \sup \{ S(y) + T(x-y) : y \in [0,x] \} \\
            (S \wed T)(x) &= \inf \{ S(y) + T(x-y) : y \in [0,x] \} \\
            S^+(x) &= \sup \{ S(y): y \in [0,x] \} \\
            S^-(x) &= \inf \{ S(y): y \in [0,x] \} \\
            \abs{S}(x) &= \sup \{ S(y): y \in [-x,x] \};
        \end{align*}
        \item for every increasing net $(T_\alp)$ in $\mc{L}_{\text{b}}(X,Y)$, we have $T_\alp \ua T$ in $\mc{L}_{\text{b}}(X,Y)$ if and only if $T_\alp (x) \ua T(x)$ in $Y$ for all $x \in X^+$.
    \end{enumerate}
\end{theorem}
\begin{proof}
    We first show that $\mc{L}_{\text{b}}(X,Y)$ is an $\L$-vector lattice, and then show that every directed set with an upper bound has a supremum.

    As in the classical case, it suffices to show that for every $S\in\mc{L}_{\text{b}}(X,Y)$, $S^+ = S \vee 0$ exists in $\mc{L}_{\text{b}}(X,Y)$. To that end, define $R : X^+ \to Y^+$ by setting, for each $x\in X^+$, $R(x) := \sup \{S(y): 0 \leq y \leq x\}$. We will use Theorem \ref{thm-pos-mapping-extension} to extend $R$ to all of $X$. The additivity of $R$ can be proven exactly as in the classical case, \cite[Theorem 1.59]{cones_and_duality}. (In particular, this is where the RDP is used.) For $\P$-homogeneity, let $x\in X^+$ and $\pi \in \P$. By Lemma \ref{lem-P-preserves-intervals}, $S([0, \pi x]) = S(\pi [0,x])$, and by linearity, $S(\pi [0,x]) = \pi S([0,x])$. Taking the supremum gives $\sup \pi S([0,x]) = \pi \sup S([0,x])$ by Lemma \ref{lem-scalar-mult-continuity-right}, and hence $R(\pi x) = \pi R(x)$.
    
    Since $Y$ is Dedekind-complete, it is $\R$-Archimedean, and thus sequentially almost $\R$-Archimedean. Now, recall that the statement of the present theorem requires that either $Y$ is sequentially $\P$-Archimedean or $\L=\L_1$. If $Y$ is sequentially $\P$-Archimedean, then by Theorem \ref{thm-P-R-arch-implies-arch-seq}, $Y$ is sequentially Archimedean and thus sequentially almost Archimedean. In this case, by Theorem \ref{thm-pos-mapping-extension}, we see that $R$ extends uniquely to a positive operator from $X$ to $Y$. If $Y$ is not sequentially $\P$-Archimedean but $\L=\L_1$, then by Remark \ref{rem-pos-extension-L_1}, we also see that $R$ extends uniquely to a positive operator from $X$ to $Y$. Hence, in both cases, $R$ extends uniquely to a positive operator $R\in \mc{L}_{\text{b}}(X,Y)$.

    To see that $R$ is an upper bound of $\{0,S\}$ in $\mc{L}_{\text{b}}(X,Y)$, note that for all $x\in X^+$, $\{0,S(x)\} \subseteq S([0,x])$ and so $\{0,S(x)\} \leq R(x)$. Now let $T \in \mc{L}_{\text{b}}(X,Y)$ be an upper bound of $\{0,S\}$. Then if $0 \leq y \leq x$ in $X$, we have $S(y) \leq T(y) \leq T(x)$, which implies that $R(x) = \sup S([0,x]) \leq T(x)$. Thus $R$ is the supremum of $\{0,S\}$ in $\mc{L}_{\text{b}}(X,Y)$.

    This shows that $\mc{L}_{\text{b}}(X,Y)$ is a lattice, and thus $\mc{L}_{\text{b}}(X,Y) = \mc{L}_\text{r}(X,Y)$. The validity of the formulas for the absolute value and supremum of two elements of $\mc{L}_{\text{b}}(X,Y)$ can be proven exactly as in the classical case.

    We now show that $\mc{L}_{\text{b}}(X,Y)$ is Dedekind-complete. Since $\mc{L}_{\text{b}}(X,Y)$ is a lattice, it suffices to show that every non-empty directed subset of $\mc{L}_{\text{b}}(X,Y)$ with an upper bound has a supremum. In doing so, we also prove the last statement of the present theorem.
    
    Let $\mc{A} \subseteq \mc{L}_{\text{b}}(X,Y)^+$ be directed and bounded above, say by $R \in \mc{L}_{\text{b}}(X,Y)$. Then, for all $x\in X^+$, the set $\{ T(x) : T\in\mc{A} \}$ is bounded above by $R(x)$. Since $Y$ is Dedekind-complete, the set $\{ T(x) : T\in\mc{A} \}$ has a supremum, which we denote by $S(x)$. We will use Theorem \ref{thm-pos-mapping-extension} to extend $S$.
    
    To show that $S$ is additive, let $x,y \in X^+$. Then $S(x+y) = \sup_{T\in\mc{A}}T(x+y) = \sup_{T\in\mc{A}}(T(x)+T(y)) \leq \sup_{T\in\mc{A}}T(x) + \sup_{T\in\mc{A}}T(y) = S(x)+S(y)$. For the reverse inequality, let $T,U \in \mc{A}$. Since $\mc{A}$ is directed, there is a $V \in \mc{A}$ such that $T,U \leq V$. Hence $T(x) + U(y) \leq V(x)+V(y) = V(x+y) \leq S(x+y)$. Taking the supremum over all $T \in \mc{A}$ and then over all $U \in \mc{A}$ gives $S(x) + S(y) \leq S(x+y)$, concluding the proof that $S$ is additive. For $\P$-homogeneity, let $x \in X^+$ and $\pi \in \P$. By Lemma \ref{lem-scalar-mult-continuity-right}, $S(\pi x) = \sup_{T\in\mc{A}} T(\pi x) = \sup_{T\in\mc{A}} \pi T(x) = \pi S(x)$. By Theorem \ref{thm-pos-mapping-extension} (again, using the assumption that either $Y$ is sequentially $\P$-Archimedean or $\L=\L_1$), $S$ extends uniquely to a positive operator from $X$ to $Y$, which we also denote by $S$. It is immediate that $S$ is the supremum of $\mc{A}$ in $\mc{L}_{\text{b}}(X,Y)$, and so $\mc{L}_{\text{b}}(X,Y)$ is Dedekind-complete.
\end{proof}

Since $\L$ is a Dedekind-complete Archimedean $\L$-vector lattice, we have the following corollary.

\begin{corollary}
    Let $X$ be a directed $\L$-module with the RDP. Then $X^\sim$ is a Dedekind-complete $\L$-vector lattice, and its lattice operations are given by the Riesz-Kantorovich formulas. \qed
\end{corollary}

\printbibliography

\end{document}